\documentclass[11pt,reqno]{amsart}

\usepackage[dvipsnames]{xcolor}
\usepackage{tikz}
\usepackage{stmaryrd}
\usetikzlibrary{matrix,arrows,decorations.pathmorphing}
\usepackage{graphicx}
\usepackage{mathpazo}
\usepackage{setspace}
\usepackage{euler}
\usepackage{amssymb}
 \usepackage[stable]{footmisc}
\usepackage{amsmath,mathtools}
\usepackage{fullpage}
\usepackage{caption}
\usepackage{amsthm}
\usepackage{mathrsfs}
\usepackage{thmtools}
\usepackage{blkarray}
\usepackage{multirow}
\usepackage{picinpar} 
\usepackage{tikz-cd}
\usepackage{color}
\usepackage{verbatim}
\usepackage{hyperref}
\hypersetup{colorlinks=true, linkcolor=black, citecolor = OliveGreen, urlcolor = black}
\usepackage{amssymb}
\usepackage{amsmath}
\usepackage{enumitem,colonequals}

\usepackage{color}
\usepackage{mathrsfs}
\usepackage[all]{xy}
\usepackage{comment}
\usepackage{mathpazo}
\usepackage{euler}

\usepackage[margin=1.in]{geometry}

\newcommand{\mZ}{\mathbb{Z}}

\newcommand{\mC}{\mathbb{C}}

\newcommand{\mA}{\mathbb{A}}

\newcommand{\mP}{\mathbb{P}}

\newcommand{\cO}{\mathcal{O}}

\newcommand{\wt}[1]{\widetilde{#1}}

\newcommand{\cL}{\mathcal{L}}

\usepackage[OT2,T1]{fontenc}
\DeclareSymbolFont{cyrletters}{OT2}{wncyr}{m}{n}
\DeclareMathSymbol{\Sha}{\mathalpha}{cyrletters}{"58}

\DeclareMathSymbol{\Sha}{\mathalpha}{cyrletters}{"58}

\newcommand{\pwr}[1]{ \left( #1 \right) }

\newcommand{\sI}{{\mathscr I}}

\newcommand{\sL}{{\mathscr L}}

\newcommand{\sO}{{\mathscr O}}

\RequirePackage{mathrsfs} 

\theoremstyle{theorem}
\numberwithin{equation}{subsection}
\newtheorem{thmx}{\text{Theorem}}

\newtheorem{corox}[thmx]{\text{Corollary}}
\newtheorem{theorem}[subsection]{Theorem}

\newtheorem{lemma}[subsection]{Lemma}
\newtheorem{corollary}[subsection]{Corollary}

\newtheorem{prop}[subsection]{Proposition}

\numberwithin{equation}{subsection}

\theoremstyle{definition}
\newtheorem{definition}[subsection]{\text{Definition}}
\newtheorem{question}[subsection]{Question}

\theoremstyle{remark}

\newtheorem{remark}[subsection]{Remark}

\newtheorem{example}[subsection]{Example}

\numberwithin{equation}{section} \numberwithin{figure}{section}

\DeclareMathOperator{\Gal}{Gal} 
\DeclareMathOperator{\Aut}{Aut} \DeclareMathOperator{\Spec}{Spec}

\DeclareMathOperator{\pr}{pr}

\DeclareMathOperator{\codim}{codim}

\newcommand{\cdef}[1]{\textsf{\textit{#1}}}

\newcommand{\inv}{^{-1}}

\newcommand\fm{\mathfrak{m}}

\newcommand\mfm{\mathfrak{m}}

\renewcommand{\leq}{\leqslant}

\renewcommand{\le}{\leqslant}
\renewcommand{\geq}{\geqslant}
\renewcommand{\ge}{\geqslant}

\DeclareMathOperator{\et}{{\acute{e}t}}

\DeclareMathOperator{\Top}{top}

\makeatletter
\@namedef{subjclassname@1991}{\emph{2020} Mathematics Subject Classification}
\makeatother

\begin{document}

\title{Integral points on varieties with infinite \'etale fundamental group}	

\author{Niven T. Achenjang}
\address{Niven T. Achenjang \\
	Department of Mathematics \\
	Massachusetts Institute of Technology \\
	Cambridge, MA 02139 \\
	United States}
\email{nivent@mit.edu}

\author{Jackson S. Morrow}
\address{Jackson S. Morrow \\
	Department of Mathematics \\
	University of California, Berkeley \\
	Berkeley, CA 94720\\
	United States}
\email{jacksonmorrow@berkeley.edu}

\begin{abstract}
We study integral points on varieties with infinite \'etale fundamental groups. 
More precisely,  for a number field $F$ and $X/F$ a smooth projective variety, we prove that for any geometrically Galois cover $\varphi\colon Y \to X$ of degree at least $2\dim(X)^2$, 
there exists an ample line bundle $\mathscr{L}$ on $Y$ such that for a general member $D$ of the complete linear system $|\mathscr{L}|$, $D$ is geometrically irreducible 
and any set of $\varphi(D)$-integral points on $X$ is finite. 
We apply this result to varieties with infinite \'etale fundamental group to give new examples of irreducible, ample divisors on varieties for which finiteness of integral points is provable.
\end{abstract}


\date{\today}
\maketitle



\section{\bf Introduction}
\label{sec:intro}
The main goal of this work is to provide new examples of irreducible divisors $D$ on varieties $X$ over a number field $F$ for which any set of $D$-integral points on $X$ is finite. 

We begin our discussion of integral points on varieties with the work of Siegel \cite{Siegel:Integral}. 
To state his results, let $F$ be a number field,  let $S$ be a finite set of places of $F$ containing the Archimedean places, let $C$ be an affine curve over $F$ embedded in affine space $\mA_F^m$,  and let $\wt{C}$ be the projective closure of $C$. 
Siegel proved that if $\#(\wt{C}\setminus C) > 2$ over $\overline{F}$, then $C$ has finitely many points in $\mA^m(\sO_{F,S})$. 
More generally,  this result states that for a smooth projective curve $\wt{C}$ of genus $g(\wt{C})$ over $F$, effective divisors $D_1,\dots,D_q$ on $\wt{C}$, and $D = \sum_{i=1}^q D_i$, any set of $D$-integral points on $\wt{C}$ is finite if $2 - 2g(\wt{C}) - q $ is negative.  

In \cite{CorvajaZannier:IntegralPointsCurves}, Corvaja and Zannier provided a new proof of Siegel's theorem using the Schmidt subspace theorem.  
Soon after, the same authors \cite{CorvajaZannier:IntegralPointsSurfaces} extended their proof technique to show that for $X$ a smooth projective surface over $F$ and $D = \sum_{i=1}^q D_i$ where $q\geq 4$ and $D_i$ are distinct irreducible divisors such that no three share a common point and pairs of them are subject to certain intersection product constraints,  any set of $D$-integral points on $X$ is not Zariski dense. 
In the higher dimensional setting, the works of Autissier \cite{Autissier:GeometryPointsEntire, Autissier:NonDensity}, Corvaja--Levin--Zannier \cite{CLZ}, and Levin \cite{Levin} proved that for $X$ a smooth projective variety of dimension $n$ over $F$, if $D_1,\dots,D_q$ are ample divisors such that at most $n$ of them contain a given point and $D = \sum_{i=1}^q D_i$, then any set of $D$-integral points on $X$ is not Zariski dense once $q$ is large enough relative to $n$.  Morally, these results assert that once a divisor $D$ has enough irreducible components which are in general position, any set of $D$-integral points on $X$ is not Zariski dense. 

A natural follow-up question to these results is:
\begin{question}\label{question}
For which smooth projective $F$-varieties $X$, do there exist \textit{(geometrically) irreducible} divisors $D$ on $X$ for which one can prove non-density or finiteness of $D$-integral points on $X$? 
\end{question}
As mentioned above, Siegel's theorem asserts that for a smooth projective curve $C$ of genus $\geq 1$ and any irreducible effective divisor $D$ on $C$, any set of $D$-integral points on $C$ is finite. 
The most famous answer to \autoref{question} is supplied by Faltings \cite[Corollary 6.2]{FaltingsLang1} where he proves that for $A$ an abelian variety over $F$ and $D$ an ample divisor on $A$, any set of $D$-integral points on $A$ is finite. 
Vojta \cite{Vojta:IntegralPointsSemiAbelian1, Vojta:IntegralPointsSemiAb2} extended these results to the semi-abelian setting.

In \cite{Faltings:NewApproachDiophantineApproximations}, Faltings constructed a class of irreducible divisors $D$ on $\mP^2$ for which $\mP^2$ has finitely many $D$-integral points. These divisors $D$ are realized as the branch locus of a suitably generic projection $X\to \mP^2$ from a smooth surface $X$ to $\mP^2$ and when $D$ is sufficiently singular, Faltings showed that $D$ decomposes into many irreducible components in the Galois closure of $X\to \mP^2$.  
The interest in these divisors stems from the fact that finiteness of the $D$-integral points on $\mP^2$ cannot be deduced by embedding the relevant varieties into semi-abelian varieties, even after a finite extension of $\mP^2$ which is unramified outside of $D$. Therefore, these examples provide genuinely new instances of finiteness of integral points. 
In \cite[Section 13]{Levin}, Levin removed an ampleness condition from Faltings construction, which provided some improvements on Faltings result. 
Additionally,  Zannier \cite{Zannier:IntegralPointsComplementRamification} continued the study of integral points on the complement in $\mP^n$ of the branch locus of a generic projection and proved results concerning the dimension of the Zariski closure of these integral points. 

\subsection*{Main contributions}
Our main contribution is an answer to \autoref{question} when the variety $X$ has a geometrically Galois cover of sufficiently large degree.  We refer the reader to Section \ref{sububsec:Galois} for our conventions regarding (geometrically) Galois covers and to \autoref{defn:arithmeticallyhyperbolic} for the definition of arithmetically hyperbolic. 

\begin{thmx}\label{thmx:main1A}
	Let $X$ be a smooth projective $F$-variety of dimension $n\geq 2$. 
	Suppose that $\varphi\colon Y \to X$ is a geometrically Galois cover of degree at least $2n^2$. 
	Then, there exists an ample line bundle $\sL$ on $Y$ such that for a general member $D$ of the complete linear system $|\sL|$,  $D$ is geometrically irreducible, $\varphi(D)$ is ample, and $X\setminus \varphi(D)$ is arithmetically hyperbolic; in particular, any set of $\varphi(D)$-integral points on $X$ is finite. 
\end{thmx}

Abelian varieties of dimension $\geq 2$ have geometrically Galois covers of such degree, and more generally varieties with infinite or large \'etale fundamental group (\autoref{defn:largeetale}) possess such covers. 
We provide more examples of such varieties in \autoref{exam:largefun} and \autoref{exam:largelocal}. 
As a corollary to \autoref{thmx:main1A}, we have the following. 

\begin{corox}\label{corox:main1A}
	Let $X$ be a smooth projective $F$-variety of dimension $n\geq 2$ with infinite \'etale fundamental group. Then, there exist infinitely many irreducible, ample divisors $D$ on $X$ such that $X\setminus D$ is arithmetically hyperbolic; in particular, such that any set of $D$-integral points is finite. If moreover $X(F)\neq\emptyset$, then there are infinitely many such $D$ which are geometrically irreducible.
\end{corox}


Our methods are founded in Diophantine approximation, and hence Vojta's dictionary \cite{Vojta87} allows us to deduce Nevanlinnan analogues of our arithmetic results. 

\begin{thmx}\label{thmx:main1N}
Let $X$ be a complex smooth projective variety of dimension $n\geq 2$. 
Suppose that $\varphi\colon Y \to X$ is a Galois cover of degree at least $2n^2$. 
Then, there exists an ample line bundle $\sL$ on $Y$ such that for a general member $D$ of the complete linear system $|\sL|$, $D$ is irreducible, $\varphi(D)$ is ample, and any holomorphic morphism $\mC \to X\setminus \varphi(D)$ is constant. 
\end{thmx}

\begin{corox}\label{corox:main1N}
Let $X$ be a complex smooth projective variety of dimension $n\geq 2$ with infinite fundamental group. 
Then, there exist infinitely many irreducible, ample divisors $D$ on $X$ such that any holomorphic morphism $\mC \to X\setminus D$ is constant.  
\end{corox}

The proofs of \autoref{thmx:main1A}, \autoref{thmx:main1N}, and our subsequent corollaries reduce to showing a purely algebro-geometric statement. 
To state this, let $\overline{F}$ denote an algebraic closure of $F$, and for $X$ a smooth projective $F$-variety of dimension $n$, let $X_{\overline{F}}$ denote the base change of $X$ to $\Spec(\overline{F})$.  
Our main results will follow if we can show that for $A\subset\Aut(X_{\overline F})$ a finite set of automorphisms of $X$ which acts freely on $X(\overline{F})$, 
there exists an effective, ample divisor $D$ on $X$ such that any point of $X(\overline{F})$ is contained in at most $n$ of the divisors $\sigma(D)$, as $\sigma$ varies over $A$. 
We refer the reader to \autoref{corollary:alltogether} for the precise statement. 
Once we have proved this,  our results follow from work of Levin \cite{Levin} and Heier--Levin \cite{HeierLevin:Degeneracy} and properties of varieties with infinite \'etale fundamental group.  

\subsection*{Organization}
In Section \ref{sec:conventions}, we establish conventions and recall relevant background on integral points and varieties with infinite and large \'etale fundamental group. 
We prove our main results, \autoref{thmx:main1A} and \autoref{thmx:main1N}, in Section \ref{sec:proofs}. 

\subsection*{Acknowledgements}
We thank Bjorn Poonen for sketching the argument presented in \autoref{prop:generalpositionintersection} and for helpful comments. 
We also thank Levent Alp\"oge, Lea Beneish, Aaron Levin,  and Bjorn Poonen for helpful comments on a first draft. 
During the preparation of this article, N.T.A.~was supported by MIT’s Dean of Science Fellowship, and J.S.M.~was supported by NSF MSPRF grant DMS-2202960. 
We thank the referee for their valuable comments and for pointing out an error in a previous version.

\section{\bf Conventions and Preliminaries}
\label{sec:conventions}
In this section, we establish conventions we use throughout the work and recall some definitions and concepts from algebraic geometry, integral points, and varieties with infinite and large \'etale fundamental group. 

\subsection{Fields and algebraic geometry}
We will use $F$ to denote a number field and $F'/F$ to denote a finite extension of $F$.  
Let $M_F$ denote the set of places of $F$. 
We also let $K$ denote an arbitrary field of characteristic zero. 
As usual, the notation $\overline{F}$ or $\overline{K}$ will refer to an algebraic closure of $F$ or $K$, respectively. 
A \cdef{$K$-variety} is a geometrically integral separated scheme of finite type over $\Spec(K)$. 
For $K'/K$, we will let $X_{K'}$ denote the base change of $X$ to $\Spec(K')$.

\subsubsection{Fundamental groups}
Fix an embedding $K \hookrightarrow \mC$. 
In this work, we will need two kinds of fundamental groups, namely the topological fundamental group associated to the complex analytification of $X$ and the \'etale fundamental group of $X$.

For $K$ a field of characteristic zero, $X$ a smooth projective $K$-variety,  and a geometric base point $\overline{x}\colon \Spec(\overline{K}) \to X_{\overline{K}}$, we will write $\pi_1^{\et}(X_{\overline{K}},\overline{x})$ to denote the \'etale fundamental group of $X_{\overline{K}}$.  
We will largely ignore the base change and the base point notation and simply denote this as $\pi_1^{\et}(X)$. 
When $X$ is a smooth projective complex variety with $x\colon \Spec(\mC)\to X(\mC)$ a base point, we can consider the topological fundamental group of the complex analytification, and we denote this fundamental group as $\pi_1^{\Top}(X(\mC),x)$. 
As before, we will normally ignore the base point from considerations. 

\subsubsection{Galois covers}\label{sububsec:Galois}
Let $X$ be a normal projective $K$-variety. 
We say that a $K$-variety $Y$ is a \cdef{finite \'etale cover} of $X$ if there exists a finite \'etale morphism $Y\to X$. 
For our purposes, a \cdef{Galois cover of $X$} refers to the data of a projective normal variety $Y$ defined over a finite extension $K'/K$ and a finite \'etale cover $Y \to X_{K'}$ such that there exists a finite group $G$ and an action $\alpha\colon G \times Y \to Y$ where the induced morphism $\alpha \times \pr_2\colon G \times Y \to Y \times_{X_{K'}} Y$ is an isomorphism.  
These conditions imply that the action morphism $\alpha$ on $Y(\overline{K})$ is free. 
Similarly, a \cdef{geometrically Galois cover of $X$} is the data of a projective normal variety $Y$ defined over $K$ and a finite \'etale cover $Y\to X$ such that $Y_{\overline K}\to X_{\overline K}$ is a Galois cover. Note that, in contrast with our notion of Galois covers, we require the morphism $Y\to X$ to be defined over $K$, but allow for automorphisms defined over an extension of $K$. 

\subsection{Integral points and arithmetic hyperbolicity}
\label{sec:integralpointarithmetichyperbolicity}
Since our results are primarily concerned with integral points on varieties, we briefly recall the construction of $(D,S)$-integral points on a normal projective $F$-variety $X$ where $D$ is an effective Cartier divisor and $S\subset M_F$ is a finite set of places of $F$ containing the Archimedean ones. 
Roughly speaking, $(D,S)$-integral points correspond to scheme-theoretic $\cO_{F,S}$-integral points on $X\setminus D$. 
More precisely,  a subset $R\subset X({F})\setminus D$ is defined to be a collection of $(D,S)$-integral points if there exist a global Weil function $(\lambda_{D,v})_{v\in M_F}$ for $D$ such that for all $v\in M_F \setminus S$,  $\lambda_{D,v}(P)\leq 0$ for all $P\in R$. 
We will normally ignore the finite set of places $S$ from the notation and simply refer to $(D,S)$-integral points as $D$-integral points.  We refer the reader to \cite[Section 1.4]{Vojta87} for the definition of global Weil functions and further discussion on integral points. 

We will also use the notion of arithmetically hyperbolic from \cite{HeierLevin:Degeneracy}, which we recall below.

\begin{definition}\label{defn:arithmeticallyhyperbolic}
Let $X$ be a normal projective $F$-variety and $D$ be an effective Cartier divisor.  
We say that $X\setminus D$ is \cdef{arithmetically hyperbolic} if for every number field $F'/F$ and every finite set of places $S$ of $F'$ containing the Archimedean ones,  the sets of $F'$-rational $(D,S)$-integral points on $X$ are always finite.
\end{definition}

\subsection{Varieties with infinite and large \'etale fundamental group}
To begin this section, we prove a lemma about the existence of Galois covers with arbitrarily large degree for a normal projective $K$-variety such that $X_{\overline{K}}$ has infinite \'etale fundamental group.  



\begin{lemma}\label{lemma:largecover}
Let $X$ be a smooth projective $K$-variety such that $X_{\overline{K}}$ has infinite \'etale fundamental group and $X(K)\neq\emptyset$.  
For any $d\geq 1$, there exists a geometrically Galois cover $ Y \to X$ (defined over $K$) of degree $\geq d$. 
\end{lemma}

\begin{proof}
Since $X_{\overline{K}}$ has infinite \'etale fundamental group, there exists a finite \'etale cover $X'\to X_{K'}$ of degree $\geq d$ defined over some finite extension $K'/K$, which corresponds to a quotient of $\pi_1^{\et}(X)$ of cardinality $\geq d$. By taking Galois closures (see e.g. \cite[\href{https://stacks.math.columbia.edu/tag/0BN2}{Tag 0BN2}]{stacks-project}) we may assume that $X'\to X_{K'}$ is Galois. Then, \cite[Proof of Lemma 5.2(1)]{Harari:WAandFundGrps} (see also \cite{MO:largeenough}) shows how to construct a finite \'etale cover $Z\to X$, defined over $K$, which is a $K$-form of $X'_{\overline K}\to X_{\overline K}$, and so is a geometrically Galois of degree $\ge d$.
\end{proof}

A natural class of varieties with infinite \'etale fundamental group are provided by varieties with large \'etale fundamental group in the sense of Kollar \cite{Kollar:ShafarevichMaps}.
We recall the definition below. 

\begin{definition}\label{defn:largeetale}
Let $X$ be a normal projective $K$-variety.  The \'etale fundamental group of $X$ is \cdef{large} if for any closed positive-dimensional integral subvariety $Y$ of $X$ with normalization $f\colon \wt{Y}\to Y$, the image of the induced map $\pi_1^{\et}(\wt{Y}) \to \pi_1^{\et}(X)$ is infinite. 
\end{definition}

Clearly, a variety with large \'etale fundamental group has an infinite \'etale fundamental group, and hence we have an immediate corollary of \autoref{lemma:largecover}. 

\begin{corollary}\label{coro:largecover}
Let $X$ be a smooth projective $K$-variety such that $X_{\overline{K}}$ has large \'etale fundamental group and $X(K)\neq\emptyset$.  
For any $d\geq 1$, there exists a geometrically Galois cover $ Y \to X$ (defined over $K$) of degree $\geq d$. 
\end{corollary}

Below, we describe examples of varieties with large \'etale fundamental group. 

\subsubsection*{Relation to topological fundamental group of complex analytification}
After fixing base points $x\in X(\mC)$, we have the group homomorphism
\[
\iota_X\colon \pi^{\Top}_1(X(\mC),x) \to \pi_1^{\et}(X,x)
\]
which identifies the \'etale fundamental group with the profinite completion of the topological one. 
Let $\wt{X}$ denote the topological cover of $X(\mC)$ corresponding to $\ker(\iota_X)$. 

\begin{prop}[\protect{\cite[Proposition 2.12.3]{Kollar:ShafarevichMaps},\cite[Proposition 1.3]{BrunebarbeMaculan:IntegralPointsLarge}}]\label{prop:complexlarge}
Suppose that $X$ is a proper $K$-variety. 
Then the \'etale fundamental group of $X$ is large if and only if the complex analytic space $\wt{X}$ does not contain positive-dimensional compact complex analytic subspaces. 
\end{prop}

\autoref{prop:complexlarge} gives us a useful criterion for determining when a proper variety has large \'etale fundamental group, and one particularly useful situation to consider is when $\wt{X}$ is the universal cover of $X(\mC)$ i.e., when $\iota_X$ is injective.  
We note that this holds when $\pi_1^{\Top}(X(\mC),x)$ is linear, and this result allows us to identify two classes of proper varieties with large \'etale fundamental group. 

\begin{example}\label{exam:largefun}
The \'etale fundamental group of $X$ is large when $X$ is an abelian variety and when $X$ is the quotient of a bounded symmetric domain in $\mC^n$ by a torsion-free co-compact lattice of its biholomorphism group. 
\end{example}

\begin{example}
Any variety which admits a finite morphism to a variety with large \'etale fundamental group will also have large \'etale fundamental group.
\end{example}

\begin{example}\label{exam:largelocal}
Another source of varieties with large \'etale fundamental groups comes from those which possess a large local system. 
Recall from \cite{Brunebarbe:IncreasingHyperbolicity}, a local system $\cL$ on $X(\mC)$ with coefficients in some field is \cdef{large} if given any non-constant morphism $f\colon Y\to X$ with $Y$ a normal irreducible complex variety, the local system $f^*\cL$ has infinite monodromy. 
Moreover, the \'etale fundamental group of a complex variety carrying a large local system is large.

It is well-known (see e.g., \cite[Proposition 4.3]{Brunebarbe:IncreasingHyperbolicity}) that complex algebraic varieties admitting a (graded-polarizable) variation of $\mZ$-mixed Hodge structure with finite period map have a large local system, and hence have large \'etale fundamental group. 
\end{example}

\section{\bf Proof of \autoref{thmx:main1A} and \autoref{thmx:main1N}}
\label{sec:proofs}
In this section, we prove our \autoref{thmx:main1A} and \autoref{thmx:main1N}. 
First, we prove several lemmas concerning the behavior of divisors under the action of a finite collection of automorphisms.  

%

\begin{lemma}\label{lem:ample}
Let $X$ be a projective $K$-variety, and fix some $d\in\mZ_{\geq 0}$. 
    There exists an ample line bundle $\sL$ on $X$ such that $H^1(X,\sL^n\otimes\sI_{S/X})=0$ for any finite subscheme $S\subset X$ of degree at most $d$ (i.e., for which $\dim_KH^0(S,\sO_S)\le d$) and any $n\geq 1$.
\end{lemma}

\begin{proof}
    This is essentially \cite[Lemma 2.1]{Poonen:BertiniFinite}. 
    Let $X\hookrightarrow\mP^N$ be a closed embedding. For any $m\ge1$, there is a commutative diagram
    \[
    \begin{tikzcd}
        H^0(\mP^N,\sO(m))\ar[d]\ar[r, "(1)"]&H^0(S,\sO_S)\ar[d, equals]\\
        H^0(X,\sO_X(m))\ar[r, "(2)"]&H^0(S,\sO_S)\ar[r]&H^1(X,\sI_{S/X}(m))\ar[r]&H^1(X,\sO_X(m))
    \end{tikzcd}
    \]
    with bottom row exact.
    For $m\ge d-1\ge\dim_K H^0(S,\sO_S)-1$, \cite[Lemma 2.1]{Poonen:BertiniFinite} tells us that (1) above is surjective, and so (2) must be surjective as well. 
    By Serre vanishing, there exists a constant $C$ (depending only on the embedding $X\hookrightarrow\mP^N$) such that $H^1(X,\sO_X(m))=0$ if $m\ge C$. Combining this with surjectivity of (2) shows that
    \[H^1(X,\sI_{S/X}(m))=0\text{ for all }m\ge\max(C,d-1)=:M.\]
    In particular, if we take $\sL=\sO_X(M)$, then $H^1(X,\sL^n\otimes\sI_{S/X})=H^1(X,\sI_{S/X}(nM))=0$ for all $n\ge1$.
\end{proof}

\begin{lemma}\label{lem:reg}
    Let $(R,\mfm)$ be a noetherian regular local ring, and consider elements $f_1,\dots,f_k\in\mfm$. Then, $S:=R/(f_1,\dots,f_k)$ is a regular local ring of dimension $\dim R-k$ if and only if
    \[f_i\not\in\mfm^2+(f_1,\dots,f_{i-1})\]
    for all $i\ge1$.
\end{lemma}
\begin{proof}
This is essentially \cite[Proposition 22 in Section IV.D.2]{Serre:LocalAlgebra}. 
\end{proof}

\begin{lemma}\label{lem:linearly-dependent-codim}
	Let $V$ be an $n$-dimensional $K$-vector space. Implicitly identify $V$ with the affine scheme $\mA(V)=\Spec\operatorname{Sym}(V^\vee)$. For any $k\ge1$, consider the Zariski closed subset
	\[Z_k:=\left\{(v_1,\dots,v_k)\in V^k:v_1,\dots,v_k\text{ are $K$-linearly dependent}\right\}.\]
	If $k\le n$, then $\codim_{V^k}(Z_k)=(n+1)-k$. Otherwise, $Z_k=V^k$.
\end{lemma}
\begin{proof}
	The claim is obvious when $k>n$, so assume $k\le n$. For any nonempty $I\subsetneq\{1,\dots,k\}$, let $p_I\colon V^k\to V^{\#I}$ denote the projection onto the coordinates contained in $I$. To each such $I$, with $\#I=:i$, we attach the subset
	\begin{align*}
		Y_I &:=\left\{\vec v=(v_1,\dots,v_k)\in V^k:p_I(\vec v)\not\in Z_i,\text{ but }p_{I\cup\{j\}}(\vec v)\in Z_{i+1}\text{ for all }j\not\in I\right\}
		.
	\end{align*}
	The image of $Y_I\xrightarrow{p_I}V^i$ is Zariski open in $V^i$, and the fibers of this map are all vector spaces of dimension $i(k-i)$. Thus, $\dim Y_I=\dim V^i+i(k-i)=i(n+k-i)$. One can easily check that this expression is increasing in $i$ while $i\le(n+k)/2$. Because we require $i<k$ in the formation of $Y_I$ (and because $k\le(n+k)/2$), we thus in fact conclude that
	\[\dim Y_I\le\dim Y_{\{1,\dots,k-1\}}=(k-1)(n+k-(k-1))=(k-1)(n+1).\]
	That is, each $Y_I$ is of codimension $\ge(n+1)-k$ in $V^k$. Unwinding definitions, one sees that
	\[Z_k=\{(0,\dots,0)\}\cup\left(\bigcup_{\emptyset\subsetneq I\subsetneq\{1,\dots,k\}}Y_I\right),\]
	from which the claim follows.
\end{proof}


\begin{prop}\label{prop:generalpositionintersection}
    Let $X$ be a smooth projective $\overline{K}$-variety, and let $A\subset\Aut(X)$ be a finite set of automorphisms of $X$ which acts freely on $X(\overline{K})$. Write $k:=\#A$ and $n:=\dim X$. 
    Then, there exists an ample line bundle $\sL$ and a dense open locus $U \subset |\sL|$ of its complete linear system such that for any effective divisor $D\in U$ the intersection $\bigcap_{\sigma\in A}\sigma(D)\subset X$ is regular of dimension $n-k$, where if $k > n$, then we mean that the intersection is empty. 
\end{prop}

\begin{proof}
	Let $d=k(n + 1) = \#A\cdot (\dim(X) + 1)$, and choose some ample $\sL$ on $X$ as in \autoref{lem:ample}. Let $|\sL|$ denote its complete linear system, so $D\in|\sL|$ denotes an effective divisor on $X$ such that $\sL \cong\sO(D)$. Given $D\in|\sL|$, we let $I(D):=\bigcap_{\sigma\in A}\sigma(D)\subset X$ denote the intersection of its translates. For any (closed) $p\in X$, we let $B_p\subset |\sL|$ be the locus of ``bad divisors'' $D$, i.e. those for which $p\in I(D)$, but $\sO_{I(D),p}$ is not regular local of dimension $n-k$.
    
    We claim that $B_p$ is of codimension $\ge n+1$ in $|\sL|$. For this, we will first set up some notation. Each $\sigma\in A$ induces an isomorphism
	\[    
    \sigma^*\colon \sO_{X,p}\cong\sO_{X,\sigma^{-1}(p)}
    \]
    of local rings. 
    Let $S'$ be $\bigcup_{\sigma\in A}\sigma^{-1}(p)$ endowed with its reduced closed subscheme structure; the points $\sigma\inv(p)$ are distinct as $\sigma\in A$ varies (for fixed $p$) by assumption.
	Let $\sI_{S'}$ denote the ideal sheaf of $S'$, and define $S$ to be the closed subscheme with ideal sheaf $\sI_{S'}^2$. 
	Note that $S$ and $S'$ have the same underlying set.

	We claim that $h^0(S,\sO_S)=k(n+1)=d$.  
	Indeed, $\sO_S$ is a skyscraper sheaf supported at $\sigma^{-1}(p)$ as $\sigma$ varies over $A$ with stalks $\sO_{S,q} = \sO_{X,q}/\fm_{q}^2$, where $q=\sigma\inv(p)$ and $\fm_{q}$ is the maximal ideal of $\sO_{X,q}$. 
	The $\overline K$-dimension of the local ring $\sO_{S,q}$ is
	\[
		\dim_{\overline K}\sO_{S,q} = \dim_{\overline K}(\sO_{X,q}/\fm_q)+\dim_{\overline K}(\fm_q/\fm_q^2) = \dim_{\overline K}(\overline K) + \dim\sO_{X,q} = 1+n,
	\]
	where $n=\dim X=\dim\sO_{X,q}=\dim_{\overline K}(\fm_q/\fm_q^2)$ since $q$ is a closed point on the smooth $\overline K$-scheme $X$.
	Since $A$ acts freely on $X(\overline{K})$ and $k = \#A$, we have that $h^0(S,\sO_S)=k(n+1)$.

	By our choice of $\sL$, the natural map
	\[    
	F\colon H^0(X,\sL)\to H^0(S,\sO_S)=\bigoplus_{\sigma\in A}\frac{\sO_{X,\sigma^{-1}(p)}}{\mfm_{\sigma^{-1}(p)}^2}\cong\bigoplus_{\sigma\in A}\frac{\sO_{X,p}}{\mfm_p^2}\cong\pwr{\frac{\sO_{X,p}}{\mfm_p^2}}^k
	\]
	 is surjective, where the second-to-last isomorphism is induced by the ${\sigma^*}^{-1}$s. 
	 By \autoref{lem:reg}, a nonzero section $s\in H^0(X,\sL)$ cuts out a divisor belonging to $B_p$ if and only if $F(s)=(f_1,\dots,f_k)\in\pwr{\frac{\sO_{X,p}}{\mfm_p^2}}^k$ satisfies $f_i\in\mfm_p/\mfm_p^2$ for all $i$, and $f_j\in\mfm_p^2+(f_1,\dots,f_{j-1})$ for some $j$. This is the case if and only if $f_1,\dots,f_k$ lie in
	 \[Z:=\left\{(f_1,\dots,f_k)\in\left(\frac{\mfm_p}{\mfm_p^2}\right)^k:f_1,\dots,f_k\text{ are $\overline K$-linearly dependent}\right\}.\]

	 By \autoref{lem:linearly-dependent-codim} applied to $V=\mfm_p/\mfm_p^2$, $Z$ is of codimension $\max\{(n+1)-k,0\}$ in $(\mfm_p/\mfm_p^2)^k$. 
	 Since $\mfm_p/\mfm_p^2$ is of codimension $1$ in $\sO_{X,p}/\mfm_p^2$, we see that $Z$ is of codimension $\max\{n+1,k\}\ge n+1$ in $(\sO_{X,p}/\mfm_p^2)^k\cong H^0(S,\sO_S)$. Now, because $F$ is a linear surjection, its corresponding map on affine schemes is faithfully flat (e.g. by \cite[Proposition 1.70]{Milne:AlgGrp}), so $F^{-1}(Z)$ is of codimension $\ge n+1$ in $H^0(X,\sL)$ as well. Hence, $B_p\subset|\sL|$ is of codimension $\ge n+1$ as claimed.

	    
	To conclude,  let $B\subset X\times |\sL|$ be (the closure of) the locus of pairs $(p,D)$ with $D\in B_p$, and consider the projections
    \[
    \begin{tikzcd}[column sep = 1cm, row sep = .7cm]
    X & B \arrow[swap]{l}{p_1} \arrow{r}{p_2} & \left|\sL\right|
    \end{tikzcd}
    \]
	Note that $p_1$ is surjective with fibers of dimension $\le\dim |\sL|-(n+1)$, so $\dim B\le\dim |\sL|-1$. Thus, $p_2$ is not surjective, so the locus $U=|\sL|\setminus p_2(B)$ consisting of divisors $D$ such that $I(D)$ is regular of dimension $n-k$ is both dense and open.
\end{proof}

\begin{corollary}\label{corollary:alltogether}
Let $X$ be a smooth projective $K$-variety, and let $A\subset\Aut(X_{\overline K})$ be a finite set of geometric automorphisms of $X$ which acts freely on $X(\overline K)$. 
Write $k:=\#A$ and $n:=\dim X \geq 2$. 
There exists an ample line bundle $\sL$ on $X$ and a dense open locus $U \subset |\sL|$ of its complete linear system such that any effective divisor $D\in U$ is ample, geometrically irreducible, and the intersection $\bigcap_{\sigma\in A}\sigma(D_{\overline K})\subset X_{\overline K}$ is regular of dimension $n-k$, where if $k > n$, then we mean that the intersection is empty. 
\end{corollary}
\begin{proof}
Choose an ample line bundle $\sL$ on $X$ as in \autoref{lem:ample} with $d = k(n+1)$. 
Applying the argument from \autoref{prop:generalpositionintersection} to $X_{\overline{K}}$, we have that the locus 
$\overline{V} \subset |\sL_{\overline{K}}| \cong \mP_{\overline{K}}^{h^0(X,\sL) - 1}$ of divisors $D$ for which  $I(D):=\bigcap_{\sigma\in A}\sigma(D)\subset X_{\overline{K}}$ is regular of dimension $n - k$ is open and dense. 
By definition, the locus $\overline{V}$ is $\Gal(\overline{K}/K)$-invariant, and so it descends to an open dense locus ${V} \subset |\sL| \cong \mP_{{K}}^{h^0(X,\sL) - 1}$ over $K$. 
We note that \cite[Corollary III.7.9]{Har} tells us that any $D\in V$ is geometrically connected, 
and Bertini's theorem \cite[Cor.I.6.11(2)]{Jouanolou:Bertini} implies that there is a dense open locus $U$ in $V$ whose corresponding divisors are smooth. Divisors in $U$ are smooth and geometrically connected; hence, they are geometrically irreducible and $U\subset V$ is the dense open locus sought after. 
\end{proof}


We now prove our main theorems, restated below for the reader's convenience.

\begin{theorem}[= \autoref{thmx:main1A}]\label{thm:1Ain}
Let $X$ be a smooth projective $F$-variety of dimension $n\geq 2$. 
Suppose that $\varphi\colon Y \to X$ is a geometrically Galois cover of degree at least $2n^2$. 
Then, there exists an ample line bundle $\sL$ on $Y$ such that for a general member $D$ of the complete linear system $|\sL|$,  $D$ is geometrically irreducible, $\varphi(D)$ is ample, and $X\setminus \varphi(D)$ is arithmetically hyperbolic; in particular, any set of $\varphi(D)$-integral points on $X$ is finite. 
\end{theorem}
\begin{proof}
Let $G\subset\Aut(Y_{\overline F})$ denote the Galois group of $\varphi$.  
By \autoref{corollary:alltogether}, we can find an ample line bundle $\sL$ on $Y$ such that a general member $D$ of the complete linear system $|\sL|$ is ample, geometrically irreducible, and the intersection (in $Y_{\overline F}$) of any $\dim(Y)+1$ of the divisors $\{ \sigma(D_{\overline F}) : \sigma \in G\}$ is empty. Let $F'/F$ be a finite extension over which all $\sigma\in G$ are defined, and let $\varphi_{F'}\colon Y_{F'}\to X_{F'}$ be the base change of $\varphi$.
Since $D$ is ample, $\sigma(D_{F'})$ is as well for all $\sigma \in G$. 
Since $\varphi^{*}_{F'}(\varphi_{F'}(D_{F'})) = \sum_{\sigma\in G}\sigma(D_{F'})$ and $\#G\ge2n^2 = 2\dim(Y_{F'})^2$,  \cite[Theorem 1.4]{HeierLevin:Degeneracy} asserts $Y_{F'}\setminus \varphi^{*}_{F'}(\varphi_{F'}(D_{F'}))$ is arithmetically hyperbolic.
Since $\varphi_{F'}$ is finite \'etale, the integral Chevalley--Weil theorem \cite[$\S\, 5.1$]{Vojta87} implies that $X_{F'}\setminus \varphi_{F'}(D_{F'})$ is arithmetically hyperbolic. 
To finish, note that the natural projection $f\colon X_{F'}\to X$ is finite \'etale and that $\varphi(D)=f(\varphi_{F'}(D_{F'}))$, so a second application of integral Chevalley--Weil, this time to $f\colon X_{F'}\to X$, shows that $X\setminus \varphi(D)$ is arithmetically hyperbolic.
Finally, ampleness of $\varphi(D)$ follows from \cite[Exercise III.5.7(d)]{Har}, as $(f\circ\varphi_{F'})^*(\varphi(D))=\varphi^*(\varphi(D))_{F'}$ is ample.
\end{proof}

\begin{corollary}[= \autoref{corox:main1A}]\label{coro:1Ain}
Let $X$ be a smooth projective $F$-variety of dimension $n\geq 2$ with infinite \'etale fundamental group. Then, there exist infinitely many irreducible, ample divisors $D$ on $X$ such that $X\setminus D$ is arithmetically hyperbolic; in particular, such that any set of $D$-integral points is finite. If moreover $X(F)\neq\emptyset$, then there are infinitely many such $D$ which are geometrically irreducible.
\end{corollary}

\begin{proof}

Since $X$ has infinite \'etale fundamental group, there exists a Galois cover $\varphi\colon Y\to X_{F'}$ of degree at least $2n^2$, defined over some finite extension $F'/F$. \autoref{thm:1Ain} guarantees the existence of infinitely many geometrically irreducible, ample divisors $D'\subset Y$ such that $\varphi(D')\subset X_{F'}$ is geometrically irreducible and ample, and such that $X_{F'}\setminus\varphi(D')$ is arithmetically hyperbolic. Letting $f:X_{F'}\to X$ denote the natural projection, the integral Chevalley-Weil theorem \cite[$\S\, 5.1$]{Vojta87} tells us that, for any such $D'$, $X\setminus f(\varphi(D'))$ is arithmetically hyperbolic as well. As in the proof of \autoref{thm:1Ain}, each $f(\varphi(D'))$ is ample as a consequence of \cite[Exercise III.5.7(d)]{Har}. Because $f\circ\varphi\colon Y\to X$ is finite, the infinitely many (geometrically) irreducible, ample $D'$ on $Y$ give rise to infinitely many irreducible, ample divisors $f(\varphi(D'))$ on $X$ whose complements are arithmetically hyperbolic.  Finally, if $X(F)\neq\emptyset$, \autoref{lemma:largecover} says we can take $\varphi$ to be defined over $F$ (at the expense of making it geometrically Galois, but not necessarily Galois), and so we directly get geometrically irreducible divisors, namely $\varphi(D')$, defined over $F$ in the above application of \autoref{thm:1Ain}.
\end{proof}

\begin{theorem}[= \autoref{thmx:main1N}]\label{thm:1Nin}
Let $X$ be a complex smooth projective variety of dimension $n\geq 2$. 
Suppose that $\varphi\colon Y \to X$ is a Galois cover of degree at least $2n^2$. 
Then, there exists an ample line bundle $\sL$ on $Y$ such that for a general member $D$ of the complete linear system $|\sL|$, $D$ is irreducible, $\varphi(D)$ is ample, and any holomorphic morphism $\mC \to X\setminus \varphi(D)$ is constant. 
\end{theorem}

\begin{proof}
The proof of this statement follows identically from the proof of \autoref{thmx:main1A} except one needs to replace \cite[Theorem 1.4]{HeierLevin:Degeneracy} with its Nevanlinnan analogue (cf.~\cite[Theorem 9.11B.(a)]{Levin}), and note that any holomorphic morphism $\mC \to Y\setminus \varphi^{*}(\varphi(D))$ is constant if and only if any holomorphic morphism $\mC \to X\setminus \varphi(D)$ is constant since $\varphi$ is finite \'etale. 
\end{proof}

\begin{corollary}[= \autoref{corox:main1N}]
Let $X$ be a complex smooth projective variety of dimension $n\geq 2$ with infinite fundamental group. 
Then, there exist infinitely many irreducible, ample divisors $D$ on $X$ such that any holomorphic morphism $\mC \to X\setminus D$ is constant.  
\end{corollary}

\begin{proof}
The proof follows that of \autoref{coro:1Ain} except that \autoref{thm:1Ain} is replaced with \autoref{thm:1Nin}. 
\end{proof}

\begin{remark}
While \autoref{thmx:main1N} and \autoref{corox:main1N} contain the assumption that $X$ has dimension $n\geq 2$, the conclusions of these statements hold when $X$ has dimension $1$ by Picard's theorem.

On the other hand, most but not all of the conclusions of \autoref{thmx:main1A} and \autoref{corox:main1A} hold when $X$ has dimension $1$. 
When $X$ has dimension 1 and infinite \'etale fundamental group,  Siegel's theorem tells us that there exist infinitely many irreducible, ample divisors $D$ on $X$ such that $X\setminus D$ is arithmetically hyperbolic. 
The divisors $D$ will be \textit{geometrically} irreducible when they correspond to a $F$-rational point of $X$. If $X(F)\neq \emptyset$,  there does exist a {geometrically} irreducible, ample divisor $D$ on $X$ such that $X\setminus D$ is arithmetically hyperbolic.  However, the Mordell--Weil theorem and Faltings' theorem tell us that we will have infinitely many $F$-rational points (hence geometrically irreducible divisors) only when $X$ is an elliptic curve of positive Mordell--Weil rank over $F$. 
\end{remark}

\begin{remark}\label{rem:phi(D)-not-normal}
The irreducible divisors $\varphi(D)$ we have constructed in the proof of \autoref{thmx:main1A} cannot be normal when $\dim(X) \geq 2$. 
Suppose that $\varphi(D)$ is normal.  Since $\varphi$ is finite \'etale, $\varphi^*(\varphi(D))$ is normal and ample, and hence connected by \cite[Corollary III.7.9]{Har}.  Since $\varphi^*(\varphi(D))$ is normal and connected, it is irreducible, contradicting its construction from a union of many distinct effective divisors.
\end{remark}

  \bibliography{refs}{}

\def\cprime{$'$}
\providecommand{\bysame}{\leavevmode\hbox to3em{\hrulefill}\thinspace}
\providecommand{\MR}{\relax\ifhmode\unskip\space\fi MR }
\providecommand{\MRhref}[2]{%
  \href{http://www.ams.org/mathscinet-getitem?mr=#1}{#2}
}
\providecommand{\href}[2]{#2}
\begin{thebibliography}{{Sta}15}

\bibitem[Aut09]{Autissier:GeometryPointsEntire}
Pascal Autissier, \emph{G\'{e}om\'{e}trie, points entiers et courbes
  enti\`eres}, Ann. Sci. \'{E}c. Norm. Sup\'{e}r. (4) \textbf{42} (2009),
  no.~2, 221--239. \MR{2518077}

\bibitem[Aut11]{Autissier:NonDensity}
\bysame, \emph{Sur la non-densit\'{e} des points entiers}, Duke Math. J.
  \textbf{158} (2011), no.~1, 13--27. \MR{2794367}

\bibitem[BM22]{BrunebarbeMaculan:IntegralPointsLarge}
Yohan Brunebarbe and Marco Maculan, \emph{Counting integral points on varieties
  with large fundamental group}, Preprint, arXiv:2205.05436 (May 30, 2022).

\bibitem[Bru20]{Brunebarbe:IncreasingHyperbolicity}
Yohan Brunebarbe, \emph{Increasing hyperbolicity of varieties supporting a
  variation of {H}odge structures with level structures}, Preprint,
  arXiv:2007.12965 (July 25, 2020).

\bibitem[CLZ09]{CLZ}
Pietro Corvaja, Aaron Levin, and Umberto Zannier, \emph{Integral points on
  threefolds and other varieties}, Tohoku Math. J. (2) \textbf{61} (2009),
  no.~4, 589--601. \MR{2598251 (2011a:11125)}

\bibitem[Cu]{MO:largeenough}
Will Chen~(\url{https://mathoverflow.net/users/15242/will-chen}), \emph{Are
  {''l}arge enoug{h''} finite etale covers arithmetic?}, MathOverflow,
  URL:\url{https://mathoverflow.net/q/364290} (version: 2020-06-28).

\bibitem[CZ02]{CorvajaZannier:IntegralPointsCurves}
Pietro Corvaja and Umberto Zannier, \emph{A subspace theorem approach to
  integral points on curves}, C. R. Math. Acad. Sci. Paris \textbf{334} (2002),
  no.~4, 267--271. \MR{1891001}

\bibitem[CZ04]{CorvajaZannier:IntegralPointsSurfaces}
P.~Corvaja and U.~Zannier, \emph{On integral points on surfaces}, Ann. of Math.
  (2) \textbf{160} (2004), no.~2, 705--726. \MR{2123936}

\bibitem[Fal91]{FaltingsLang1}
Gerd Faltings, \emph{Diophantine approximation on abelian varieties}, Ann. of
  Math. (2) \textbf{133} (1991), no.~3, 549--576. \MR{1109353}

\bibitem[Fal02]{Faltings:NewApproachDiophantineApproximations}
\bysame, \emph{A new application of {D}iophantine approximations}, A panorama
  of number theory or the view from {B}aker's garden ({Z}\"{u}rich, 1999),
  Cambridge Univ. Press, Cambridge, 2002, pp.~231--246. \MR{1975455}

\bibitem[Har77]{Har}
R.~Hartshorne, \emph{Algebraic geometry}, Springer-Verlag, New York, 1977,
  Graduate Texts in Mathematics, No. 52. \MR{0463157 (57 \#3116)}

\bibitem[Har00]{Harari:WAandFundGrps}
D.~Harari, \emph{Weak approximation and non-abelian fundamental groups}, Ann.
  Sci. \'{E}cole Norm. Sup. (4) \textbf{33} (2000), no.~4, 467--484.
  \MR{1832820}

\bibitem[HL20]{HeierLevin:Degeneracy}
Gordon Heier and Aaron Levin, \emph{On the degeneracy of integral points and
  entire curves in the complement of nef effective divisors}, J. Number Theory
  \textbf{217} (2020), 301--319. \MR{4140631}

\bibitem[Jou83]{Jouanolou:Bertini}
Jean-Pierre Jouanolou, \emph{Th{\'e}oremes de {Bertini} et applications}, Prog.
  Math., vol.~42, Birkh{\"a}user, Cham, 1983 (French).

\bibitem[Kol93]{Kollar:ShafarevichMaps}
J\'{a}nos Koll\'{a}r, \emph{Shafarevich maps and plurigenera of algebraic
  varieties}, Invent. Math. \textbf{113} (1993), no.~1, 177--215. \MR{1223229}

\bibitem[Lev09]{Levin}
Aaron Levin, \emph{Generalizations of {S}iegel's and {P}icard's theorems}, Ann.
  of Math. (2) \textbf{170} (2009), no.~2, 609--655. \MR{2552103 (2010k:11116)}

\bibitem[Mil17]{Milne:AlgGrp}
J.~S. Milne, \emph{Algebraic groups}, Cambridge Studies in Advanced
  Mathematics, vol. 170, Cambridge University Press, Cambridge, 2017, The
  theory of group schemes of finite type over a field. \MR{3729270}

\bibitem[Poo04]{Poonen:BertiniFinite}
Bjorn Poonen, \emph{Bertini theorems over finite fields}, Ann. of Math. (2)
  \textbf{160} (2004), no.~3, 1099--1127. \MR{2144974}

\bibitem[Ser00]{Serre:LocalAlgebra}
Jean-Pierre Serre, \emph{Local algebra}, Springer Monographs in Mathematics,
  Springer-Verlag, Berlin, 2000, Translated from the French by CheeWhye Chin
  and revised by the author. \MR{1771925}

\bibitem[Sie14]{Siegel:Integral}
Carl~L. Siegel, \emph{\"{U}ber einige {A}nwendungen diophantischer
  {A}pproximationen [reprint of {A}bhandlungen der {P}reu\ss ischen {A}kademie
  der {W}issenschaften. {P}hysikalisch-mathematische {K}lasse 1929, {N}r. 1]},
  On some applications of {D}iophantine approximations, Quad./Monogr., vol.~2,
  Ed. Norm., Pisa, 2014, pp.~81--138. \MR{3330350}

\bibitem[{Sta}15]{stacks-project}
The {Stacks Project Authors}, \emph{\emph{{S}tacks {P}roject}},
  http://stacks.math.columbia.edu, 2015, Accessed on November 28, 2022.

\bibitem[Voj87]{Vojta87}
Paul Vojta, \emph{Diophantine approximations and value distribution theory},
  Lecture Notes in Mathematics, vol. 1239, Springer-Verlag, Berlin, 1987.
  \MR{883451 (91k:11049)}

\bibitem[Voj96]{Vojta:IntegralPointsSemiAbelian1}
\bysame, \emph{Integral points on subvarieties of semiabelian varieties. {I}},
  Invent. Math. \textbf{126} (1996), no.~1, 133--181. \MR{1408559}

\bibitem[Voj99]{Vojta:IntegralPointsSemiAb2}
\bysame, \emph{Integral points on subvarieties of semiabelian varieties. {II}},
  Amer. J. Math. \textbf{121} (1999), no.~2, 283--313. \MR{1680329}

\bibitem[Zan05]{Zannier:IntegralPointsComplementRamification}
Umberto Zannier, \emph{On the integral points on the complement of
  ramification-divisors}, J. Inst. Math. Jussieu \textbf{4} (2005), no.~2,
  317--330. \MR{2135140}

\end{thebibliography}
\bibliographystyle{amsalpha}

 \end{document}